\documentclass[a4paper]{article}

\usepackage[english]{babel}
\usepackage[utf8x]{inputenc}
\usepackage[T1]{fontenc}

\usepackage[a4paper,top=3cm,bottom=2cm,left=3cm,right=3cm,marginparwidth=1.75cm]{geometry}

\usepackage{amsmath,amssymb,amsthm}
\usepackage{graphicx}
\usepackage[colorinlistoftodos]{todonotes}
\usepackage[colorlinks=true, allcolors=blue]{hyperref}

\newcommand\cH{{\mathcal H}}

\newtheorem{theorem}{Theorem}[section]
\newtheorem{lemma}[theorem]{Lemma}

\newtheorem{claim}[theorem]{Claim}

\numberwithin{equation}{section}

\newcommand\cref[1]{Corollary~\ref{cor:#1}}

\title{A note on the Tur\'an number of a Berge odd cycle}
\author{D\'aniel Gerbner\\
\medskip 
\small Alfr\'ed R\'enyi Institute of Mathematics, Hungarian Academy of Sciences\\
\small P.O.B. 127, Budapest H-1364, Hungary.\\
\medskip
\small \texttt{gerbner@renyi.hu}
}

\begin{document}

\maketitle

\begin{abstract} In this note we obtain upper bounds on the number of hyperedges in 3-uniform hypergraphs not containing a Berge cycle of given odd length. We improve the bound given by F\"uredi and \"Ozkahya in 2017. The result follows from a more general theorem. We also obtain some new results for Berge cliques.

\end{abstract}

\section{Introduction}

We say that a hypergraph $\cH$ is a Berge copy of a graph $F$ (in short: $\cH$ is a Berge-$F$) if $V(F)\subset V(\cH)$ and there is a bijection $f:E(F)\rightarrow E(\cH)$ such that for any $e\in E(F)$ we have $e\subset f(e)$. This definition was introduced by Gerbner and Palmer \cite{gp1}, extending the well-established notion of Berge cycles and paths. Note that there are several non-uniform Berge copies of $F$, and a hypergraph $\cH$ is a Berge copy of several graphs. A particular copy of $F$ defining a Berge-$F$ is called its \textit{core}. Note that there can be multiple cores in a Berge-$F$.

We denote by $ex_r(n,\text{Berge-}F)$ 
the largest number of hyperedges in an $r$-uniform Berge-$F$-free hypergraph on $n$ vertices. There are several papers dealing with $ex_r(n,\text{Berge-}C_k)$ (e.g. \cite{furozk,gyori,gyole,gyl}) or $ex_r(n,\text{Berge-}F)$ in general (e.g. \cite{gmp,gmv,gp1,gp2,pttw}). For a short survey on this topic see Subsection 5.2.2 in \cite{gp}.

In this note we consider $ex_3(n,\text{Berge-}C_k)$. In the case $k=5$, this was first studied by Bollob\'as and Gy\H ori \cite{BGY2008}. They showed $ex_3(n,\text{Berge-}C_5)\le \sqrt{2}n^{3/2}+4.5n$. This bound was improved to $(0.254+o(1))n^{3/2}$ by Ergemlidze, Gy\H ori and Methuku \cite{egym}. For cycles of any length, Gy\H ori and Lemons \cite{gyole,gyl} proved $ex_r(n,\text{Berge-}C_k)=O(n^{1+1/\lfloor k/2\rfloor})$. The constant factors were improved by Jiang and Ma \cite{jima}, and in the case $k$ is even by Gerbner, Methuku and Vizer \cite{gmv}. In the 3-uniform case, F\"uredi and \"Ozkahya \cite{furozk} obtained better constant factors (depending on $k$). In the case $k$ is even, further improvements were obtained by Gerbner, Methuku and Vizer \cite{gmv} and by Gerbner, Methuku and Palmer \cite{gmp}.

A closely related area is counting triangles in $C_k$-free graphs. More generally, let $ex(n,H,F)$ denote the maximum number of copies of $H$ in an $F$-free graph on $n$ vertices. After some sporadic results, 
the systematic study of these problems (often called \textit{generalized Tur\'an problems}) was initiated by Alon and Shikhelman \cite{as}. Their connection to Berge hypergraphs was established by Gerbner and Palmer \cite{gp2}, who proved \[ex(n,K_r,F)\le ex_r(n,\text{Berge-}F)\le ex(n,K_r,F)+ex(n,F)\] for any $r$, $n$ and $F$.

Counting triangles in $C_k$-free graphs and counting hyperedges in Berge-$C_k$-free 3-uniform hypergraphs was handled together already by Bollob\'as and Gy\H ori \cite{BGY2008} for $C_5$, and by F\"uredi and \"Ozkahya \cite{furozk}, who proved $ex(n,K_3,C_{2k})\le \frac{2k-3}{3}ex(n,C_{2k})$ and $ex_3(n,\text{Berge-}C_{2k})\le \frac{2k}{3}ex(n,C_{2k})$. Their upper bound for $ex(n,K_3,C_{2k})$ is still the best known bound, but their other upper bound was improved to $ex_3(n,\text{Berge-}C_{2k})\le\frac{2k-3}{3}ex(n,C_{2k})$ by Gerbner, Methuku and Vizer \cite{gmv} in the case $k\ge 5$ and by Gerbner, Methuku and Palmer \cite{gmp} in the case $k=3,4$.

In the case of forbidden cycles of any odd length, the number of triangles was first studied by Gy\H ori and Li \cite{gyli}, who proved\footnote{We note that the bound is incorrectly stated in their paper \cite{gyli}.} $ex(n,K_3,C_{2k+1})\le \frac{(2k- 2)(16k -1)}{3}ex(n,C_{2k})$. It was improved independently by F\"uredi and \"Ozkahya \cite{furozk} and by Alon and Shikhelman \cite{as}. The latter had the stronger bound  $ex(n,K_3,C_{2k+1})\le \frac{16(k -1)}{3}ex(\lceil n/2\rceil,C_{2k})$. In the case $k=2$, the current best bound $ex(n, K_3, C_5) \le 0.231975n^{3/2}$ is due to Ergemlidze and Methuku \cite{em}.

F\"uredi and \"Ozkahya \cite{furozk} obtained the currently best upper bound on the Berge version by showing \begin{equation}\label{equ} ex_3(n,\text{Berge-}C_{2k+1})\le ex(n,K_3,C_{2k+1})+4ex(n,C_{2k})+ 12ex_3^{lin}(n,\text{Berge-}C_{2k+1}),\end{equation} where $ex_r^{lin}(n,\text{Berge-}F)$ denotes the largest number of hyperedges in an $r$-uniform Berge-$F$-free linear hypergraph on $n$ vertices. Recall that a linear hypergraph is one in which any two hyperedges share at most one vertex.

In this note we improve the bound (\ref{equ}). Recall that we have $ex_3(n,\text{Berge-}C_{2k+1})\ge ex(n,K_3,C_{2k+1})$, thus we cannot hope for a huge improvement, especially as $ex(n,K_3,C_{2k+1})$ might be the largest of the three terms. Indeed, the best upper bound currently known is $O(n^{1+1/k})$ for all the three terms, but the dependence of the known upper bound in $k$ is the largest for $ex(n,K_3,C_{2k+1})$ (we will state these bounds after Theorem \ref{cycle}).


Recall that in case of $C_{2k}$, the two upper bounds obtained by F\"uredi and \"Ozkahya \cite{furozk} were $ex(n,K_3,C_{2k})\le \frac{2k-3}{3}ex(n,C_{2k})$ and $ex_3(n, \text{Berge-}C_{2k})\le\frac{2k}{3}ex(n,C_{2k})$, and the Berge bound was improved in \cite{gmv,gmp} to match the generalized Tur\'an bound. Our goal would be to do the same here and get rid of the terms $4ex(n,C_{2k+1})+ 12ex_3^{lin}(n,\text{Berge-}C_{2k+1})$ in (\ref{equ}). We cannot achieve that, but we decrease these additional terms. Recall that the currently best bound for the generalized Tur\'an problem is $ex(n,K_3,C_{2k+1})\le \frac{16(k -1)}{3}ex(\lceil n/2\rceil,C_{2k})$ by Alon and Shikhelman \cite{as}. Our new upper bound on $ex_3(n,\text{Berge-}C_{2k+1})$
is larger than that bound by $ex_3^{lin}(n,\text{Berge-}C_{2k+1})$. We wonder if it is an example of a more general phenomenon and similar bounds could be obtained for other graphs. 

The way we use the linearity involves subdividing an edge $uv$, i.e.\ deleting it and adding $uw$ and $vw$ for a new vertex $w$.
Our method uses only the following two properties of $C_{2k+1}$: it can be obtained from $C_{2k}$ by subdividing an edge and deleting a vertex from $C_{2k+1}$ we obtain a path. In the next theorem we state our result in the most general form.

\begin{theorem}\label{main} Let $F$ be a connected graph obtained from $F_0$ by subdividing an edge and $F'$ be obtained from $F$ by deleting a vertex. Let $c=c(n)$ be such that $ex(n,K_{r-1},F')\le cn$ for every $n$. Then we have

\textbf{(i)}  $ex_r(n,\text{Berge-}F)\le ex(n,K_r,F)+ 2^{r-1}ex(n,F_0)+ex_r^{lin}(n,\text{Berge-}F)$,

\textbf{(ii)} $ex_r(n,\text{Berge-}F)\le \max\left\{1,\frac{2c}{r}\right\}2^{r-1}ex(n,F_0) +ex_r^{lin}(n,\text{Berge-}F)$.



\end{theorem}

 In the case $F=C_{2k+1}$ we have $F_0=C_{2k}$ and $F'=P_{2k}$, the path on $2k$ vertices. A theorem of Luo \cite{luo} shows $ex(n,K_{r-1},P_{2k})\le \frac{n}{2k-1}\binom{2k-1}{r-1}$, but what we need for the 3-uniform case is the Erd\H os-Gallai theorem \cite{Er-Ga} showing $ex(n,P_{2k})\le (k-1)n$. Using this, \textbf{(ii)} of Theorem \ref{main} gives $ex_3(n,\text{Berge-}C_{2k+1})\le \frac{8k-8}{3}ex( n,C_{2k})+ex_3^{lin}(n,\text{Berge-}C_{2k+1})$ if $k>2$. We can improve this a little bit.

\begin{theorem}\label{cycle}
If $k>2$, then $ex_3(n,\text{Berge-}C_{2k+1})\le \frac{16k-16}{3}ex(\lceil n/2\rceil,C_{2k})+ex_3^{lin}(n,\text{Berge-}C_{2k+1})$ 

\hspace{0.7truecm} $\le \left( \frac{1280k-1280}{3}\sqrt{k}\log k\right)\lceil n/2\rceil^{1+1/k}+2k n^{1+1/k}+9kn+\frac{16k-16}{3}10k^2\lceil n/2\rceil$.
\end{theorem}

The bound in Theorem \ref{cycle} is currently stronger than the bound given by \textbf{(i)} of Theorem \ref{main} for $F=C_{2k+1}$ and $r=3$. However, an improvement on $ex(n,K_3,C_{2k+1})$ would immediately improve the bound in \textbf{(i)}. Any significant improvement would make \textbf{(i)} stronger than Theorem \ref{cycle} for $F=C_{2k+1}$.

The second inequality in Theorem \ref{cycle} follows from known results. F\"uredi and \"Ozkahya \cite{furozk} proved $ex_3^{lin}(n,\text{Berge-}C_{2k+1})\le 2kn^{1+1/k} + 9kn$, and Bukh and Jiang \cite{bj} obtained the strongest bound on the Tur\'an number of even cycles by showing $ex(n,C_{2k})\le 80\sqrt{k}\log kn^{1+1/k}+10k^2n$. As we do not have good lower bounds on $ex(n,C_{2k})$, we cannot be sure that the first term is actually the larger term. However, if $ex_3^{lin}(n,\text{Berge-}C_{2k+1})$ is the larger term, then our improvement on the upper bound of $ex_3(n,\text{Berge-}C_{2k+1})$ is more significant, as we changed the constant factor of that term from 12 to 1. Obviously we have $ex_3^{lin}(n,\text{Berge-}C_{2k+1})\le ex_3(n,\text{Berge-}C_{2k+1})$, hence further improvement is impossible here.

We prove Theorem \ref{main} by combining the ideas of \cite{furozk} and \cite{as} with the methods developed in \cite{gmp,gmv}. In the next section we state some lemmas needed for the proof. We give a new proof of a lemma by Gerbner, Methuku and Palmer \cite{gmp}, and we strengthen the lemma a little bit. This strengthens results on $ex_r(n,\text{Berge-}K_k)$ for some values of $r$, $k$ and $n$. In Section 3 we prove Theorems \ref{main} and \ref{cycle}.

\section{Lemmas}
We say that a graph $G$ is red-blue if each of its edges is colored with one of the colors red and blue. For a red-blue graph $G$, we denote by $G_{red}$ the subgraph spanned by the red edges and $G_{blue}$ the subgraph spanned by the blue edges. For two graphs $H$ and $G$ we denote by $N(H,G)$ the number of subgraphs of $G$ that are isomorphic to $H$. Let $g_r(G)=|E(G_{red})|+N(K_r,G_{blue})$.

\begin{lemma}[Gerbner, Methuku, Palmer \cite{gmp}]\label{celeb2} For any graph $F$ and integers $r$ and $n$, there is a red-blue $F$-free graph $G$ on $n$ vertices, such that $ex_r(n,\textup{Berge-}F)\le g_r(G)$.

\end{lemma}

Note that an essentially equivalent version was obtained by F\"uredi, Kostochka and Luo \cite{fkl}.
The proof of Lemma \ref{celeb2} relies on a lemma about bipartite graphs (hidden in the proof of Lemma 2 in \cite{gmp}). If $M$ is a matching and $ab$ is an edge in $M$, then with a slight abuse of notation we say $M(a)=b$ and $M(b)=a$.

\begin{lemma}\label{cel2} Let $\Gamma$ be a finite bipartite graph with parts $A$ and $B$ and let $M$ be a largest matching in $\Gamma$. Let $B'$ denote the set of vertices in $B$ that are incident to $M$. Then we can partition $A$ into $A_1$ and $A_2$ and partition $B'$ into $B_1$ and $B_2$ such that for $a\in A_1$ we have $M(a)\in B_1$, and every neighbor of the vertices of $A_2$ is in $B_2$. 

\end{lemma}

Here we present a proof that is built on the same principle, but is somewhat simpler than the proof found in \cite{gmp}. Before that, let us recall the well-known notion of alternating paths. Given a bipartite graph $\Gamma$ and a matching $M$ in it, a path $P$ in $\Gamma$ is called \textit{alternating} if its first edge is not in $M$, and then it alternates between edges in $M$ and edges not in $M$, finishing with an edge not in $M$. It is well-known and easy to see that deleting the edges of $P$ from $M$ and replacing them with the edges of $P$ that were not in $M$, we obtain another matching, that is larger than $M$.

\begin{proof} First we build a set $V'\subset V(\Gamma)$ in the following way. 
Let $V_0$ be the set of vertices in $A$ that are not incident to any edges of $M$.
Then in the first step we add to $V_0$ the set of vertices in $B$ that are neighbors of a vertex in $V_0$, to obtain $V_1$. In the second step we add to $V_1$ the vertices in $A$ that are connected to a vertex in $V_1$ by an edge in $M$, to obtain $V_2$. Similarly, in the $i$th step, if $i$ is odd we add to $V_{i-1}$ the set of vertices in $B$ that are neighbors of a vertex in $V_{i-1}$, while if $i$ is even,  we add to $V_{i-1}$ the vertices in $A$ that are connected to a vertex in $V_{i-1}$ by an edge in $M$ (i.e.\ $M(b)$ for some $b\in B\cap V_{i-1}$), to obtain $V_i$. After finitely many steps, $V_i$ does not increase anymore, let $V'$ be the resulting set of vertices.

We claim that no vertex from $B\setminus B'$ can be in $V'$. Indeed, such a vertex could be reached by an alternating path from a vertex in $A$ that is not incident to $M$, thus $M$ is not a largest matching, a contradiction. 





Then let $A_2=A\cap V'$, $A_1=A\setminus A_2$, $B_2=B'\cap V'$ and $B_1=B'\setminus B_2$. A vertex in $A_2$ cannot be connected to a vertex $v$ not in $B_2$, as $v$ could be added to $V'$ then. Similarly, for a vertex $u\in A_1$, $M(u)$ has to be in $B_1$, otherwise $M(u)$ is in $B_2$ and then $u$ can be added to $V'$. 

\end{proof}

Let us briefly describe how we can apply this lemma to obtain Lemma \ref{celeb2}. We take a Berge-$F$-free $r$-uniform hypergraph $\cH$ on $n$ vertices. Let $A$ be the set of hyperedges in $\cH$ and $B$ be the set of sub-edges of these hyperedges (by edge and sub-edge we always mean an edge of size two, i.e.\ a pair of vertices). We connect $a\in A$ to $b\in B$ if $a\supset b$. Let $\Gamma$ denote this auxiliary bipartite graph. Let $M$ be an arbitrary largest matching and $B'$ be the vertices of $B$ incident to the edges in $M$. It is easy to see that the elements of $B'$ form an $F$-free graph which we call $G$. Indeed, otherwise $M$ defines the bijection between a copy of $F$ and hyperedges in $\cH$ to form a Berge-$F$.

Now we apply Lemma \ref{cel2} to $\Gamma$ and $M$. We define a red-blue coloring of $G$ by taking the edges of $G$ in
$B_1$ to be the red edges, and the edges of $G$ in $B_2$ to be the blue edges. We have $|\cH|=|A_1|+|A_2|=|B_1|+|A_2|=|E(G_{red})|+|A_2|$. As hyperedges in $A_2$ have all their neighbors in $B_2$, they each contain a blue $K_r$, which is distinct from the other blue $r$-cliques obtained this way, showing $|A_2|\le N(K_r,G_{blue})$.

Let us remark here that Lemma \ref{cel2} also gives some information on the structure of $G$. If there is $a\in A_1$ that has a neighbor $b\in B\setminus B'$, then we could obtain another matching $M'$ by changing the neighbor of $a$ to $b$, i.e.\ $M'(a)=b$ and if $a'\neq a$, then $M'(a')=M(a')$. Then $B'$ is replaced by $B''=B'\setminus \{M(a)\}\cup \{b\}$. In this case the same partition of $A$ into $A_1$ and $A_2$, and the partition of $B''$ into $B_2$ and $B''\setminus B_2$ satisfies Lemma \ref{cel2}. This means for $G$ that we can delete the (red) edge $M(a)$ and replace it with the edge $b$, to obtain another $F$-free graph.

If on the other hand the vertices in $A_1$ have all their neighbors in $B'$, then we could recolor the red edges to blue. Therefore, in $G$ we can delete an edge
and add another edge so that the resulting graph is still $F$-free. Let $\alpha=\alpha_{F,n}$ be the largest value of $g_r(G')$, where $G'$ is an $n$-vertex $F$-free blue-red graph.
Assume that each $n$-vertex $F$-free blue-red graph $G'$ with $g_r(G')=\alpha$ is not monoblue and we cannot delete an edge and add another edge to $G'$ so that the resulting graph is still $F$-free. Then by the above, $G$ cannot be one of these graphs, thus $ex_r(n,\textup{Berge-}F)\le g_r(G)<\alpha$.
This is usually a negligible improvement, as we often do not even know the order of magnitude. 

However, if $F=K_k$, Gerbner, Methuku and Palmer \cite{gmp} proved that $\alpha_{K_k,n}=\max\{g_r(T_B(n,k-1)),g_r(T_R(n,k-1))\}$, where $T_B(n,k-1)$ is the monoblue Tur\'an graph $T(n,k-1)$ and $T_R(n,k-1)$ is the monored Tur\'an graph $T(n,k-1)$.
We mention without going into the details that their proof also shows that for any other graphs $G$ we have $g_r(G)<\alpha_{K_k,n}$. As we cannot delete an edge from $T(n,k-1)$ and add another edge to obtain a $K_k$-free graphs, we do have an improvement. For example, if $r=4$ and $k=5$, then the result in \cite{gmp} determines $ex_4(n,\textup{Berge-}K_k)$ for $n\ge 11$. For $n=10$,  $T(10,4)$ has 36 copies of $K_4$ and 37 edges. Therefore, (as $ex(n,K_r,F)$ is a lower bound on $ex_r(n,\textup{Berge-}F)$), we have $36\le ex_4(n,\textup{Berge-}K_k)\le 37$. With our new observation, we know $ex_4(n,\textup{Berge-}K_k)=36$.

\section{Proof of Theorems \ref{main} and \ref{cycle}}


Let $\cH$ be a Berge-$F$-free $r$-graph on $n$ vertices. We say that an edge $uv$ with $u,v\in V(\cH)$ is $t$-heavy if $u,v$ are contained together in exactly $t$ hyperedges. First we will build a linear subhypergraph $\cH_1$ in a greedy way: if we can find a hyperedge $H$ that does not share an edge with any hyperedge in $\cH_1$, we add $H$ to $\cH_1$, and then repeat this procedure. By definition, $\cH_1$ is linear. Let $\cH_2$ consist of the remaining hyperedges. Note that $|\cH|=|\cH_1|+|\cH_2|\le ex_r^{lin}(n,\text{Berge-}F)+|\cH_2|$,
and the remainder of the proof is for proving the needed upper bound on $|\cH_2|$.

 We build an auxiliary bipartite graph $\Gamma$ in the usual way: let $A$ be the set of hyperedges in $\cH_2$ and $B$ be the set of sub-edges of these hyperedges. We connect $a\in A$ to $b\in B$ if $a\supset b$. We will let $M$ be a largest matching in $\Gamma$, however, we do not choose $M$ arbitrarily. Let $M_0$ be an
 arbitrary 
 largest matching in $\Gamma$.
 Let $B'$ be the set of vertices in $B$ that are incident to some edge of $M_0$ and $A_0$ denote the set of vertices in $A$ that are incident to some edge of $M_0$.
Now a hyperedge $a\in A_0$ contains a sub-edge $M_0(a)$, at least one sub-edge $b_0$ shared with a hyperedge in $\cH_1$, maybe some sub-edges that are matched to some other $a'\in A$, and maybe some other sub-edges $b\in B\setminus B'$. We have the option to replace in $M_0$ the edge between $a$ and $M_0(a)$ with any of the edges of $\Gamma$ between $a$ and an unused sub-edge of $a$, to obtain another largest matching. We will build a largest matching $M$, that contains the same vertices ($A_0$) from $A$ as $M_0$.
 
 For $a\in A_0$, we pick $M(a)$ to be one of the sub-edges $b\in B$ of $a$ (potentially we let $M(a)=M_0(a)$) in the following way: $M(a)$ should share exactly one vertex with $b_0$ (where $b_0$ is a sub-edge that is also a sub-edge of a hyperedge in $\cH_1$) if possible. We go through the hyperedges greedily; as long as there is a hyperedge $a\in A_0$ such that $M_0(a)$ can be changed in this way, we execute the change (it is possible that $M_0(a)$ cannot be changed originally, but later a sub-edge of $a$ that is $M_0(a')$ becomes free to use, when $M(a')$ is chosen to be different from $M_0(a')$). This process finishes after finitely many (at most $|A_0|$) steps, as we change $M_0(a)$ to $M(a)$ at most once for every $a\in A_0$. After this, we rename the unchanged $M_0(a)$ to $M(a)$. 
 
 The resulting matching $M$ has the following property: for every $a\in A_0$, $a$ shares a sub-edge $b_0$ with a hyperedge in $\cH_1$, such that that either $M(a)$ shares exactly one vertex with $b_0$, or all the sub-edges of $a$ sharing exactly one vertex with $b_0$ are $M(a')$ for some $a'\in A_0$.


Now we can apply Lemma \ref{cel2} to $\Gamma$ and $M$ to obtain $A_1,A_2,B_1,B_2$. Let us call the elements of $B_1$ \textit{red} edges and the elements of $B_2$ \textit{blue} edges. Let $G$ be the graph consisting of all the red and blue edges. Then $G$ is obviously $F$-free.

Let us now take a random partition of $V(\cH)$ into $V_1$ and $V_2$. 
For every $a\in A_0$, we look at $b=M(a)$. If the two vertices of $b$ are in one part, and all the other vertices of $a$ are in the other part, we keep $a$, otherwise we delete it. Let $A^*$ denote the set of elements in $A$ that are not deleted (note that elements in $A\setminus A_0$ are never deleted, thus are in $A^*$). Let $G'$ be the graph consisting of the elements of $B'$ that are connected by an edge in $M$ to an element of $A^*$. Then $G'$ is obviously $F$-free, as it is a subgraph of $G$.

\begin{claim}
$G'$ is $F_0$-free, where $F_0$ is any graph for which F can be obtained from
$F_0$ by subdividing an edge of $F_0$.
\end{claim}

\begin{proof} Let us assume we are given a copy $Q$ of $F_0$ in $G'$ such that $uv$ is the edge that needs to be subdivided to obtain $F$. Observe that there is no edge between $V_1$ and $V_2$ in $G'$, thus $Q$ is in one of them, say $V_1$. Let $w$ be a vertex of $M(uv)$ with $u\neq w\neq v$, then $w\in V_2$, thus $w$ is not in $Q$. 

We say that a hyperedge $H$ in $\cH$ is \textit{good} if $H$ contains $u$ and $w$ for some $w\in M(uv)\setminus\{u,v\}$ and $H$ is not $M(e)$ for any edge $e$ of $Q$. If there is a good hyperedge, then we build a Berge-$F$ with the following core: we subdivide $uv$ with $w$. For each edge $e$ of this core we assign $M(e)$ except for $uw$ (where we assign $H$) and $vw$ (where we assign $M(uv)$). This way we obtain a Berge-$F$, a contradiction.

$M(uv)$ shares at least one sub-edge with a hyperedge $H\in \cH_1$. If the sub-edge shares exactly one vertex with $uv$, then $H$ is good and we are done. Thus every sub-edge of $M(uv)$ shared with a hyperedge in $\cH_1$ has to contain none or both of $u$ and $v$. In both cases, when we tried to change $M_0(M(uv))$ when constructing $M$, we failed, because all such edges are matched to some other hyperedges of $\cH_2$. In particular, $uw$ is $M(a)$ for some $a\in A_0$ and for some $w\in M(u,v)\setminus\{u,v\}$. Observe that $w$ is in $V_2$, thus $M(a)$ has vertices from both parts $V_1$ and $V_2$, hence $a$ cannot be in $A^*$ by the definition of $A^*$. This implies $a$ is good, finishing the proof.

\end{proof}

The above claim implies $G'$ has at most $ex(n,F_0)$ edges. 
For an arbitrary $a\in A$, the probability that $a$ is in $A^*$ is 
at least $1/2^{r-1}$. Let $S$ be any subset of $A$, then we have that the expected value of the number of hyperedges in $A^*\cap S$ is at least $|S|/2^{r-1}$, thus there is a partition with $|A^*\cap S|\ge |S|/2^{r-1}$. 




There are $|B_1|=|A_1|$ red edges in $G$, and there is a random partition where at least $|A_1|/2^{r-1}$ elements of $A_1$ are undeleted, hence there are at least $|A_1|/2^{r-1}$ red edges in $G'$. This implies $|A_1|/2^{r-1}\le ex(n,F_0)$. Hence there are at most $2^{r-1}ex(n,F_0)$ red edges altogether. 
For the total number of edges in $G$ we can use the same argument: there is a random partition where at least $|A_0|/2^{r-1}$ hyperedges in $A_0$ are undeleted, thus for the $G'$ defined by that partition,  we have $|A_0|=|E(G)|\le 2^{r-1}|E(G')|\le 2^{r-1}ex(n,F_0)$.

Observe that we have $|\cH_2|=|A_1|+|A_2|\le |A_1|+N(K_r,G_{blue})\le |A_1|+ex(n,K_r,F)$, hence we are done with the proof of \textbf{(i)}.

Note that $G$ is not necessarily $F_0$-free, but it is $F$-free. Let $m$ be the number of blue edges in $G$, then $G$ has at most $2^{r-1}ex(n,F_0)-m$ red edges. An argument of Gerbner, Methuku and Vizer \cite{gmv} bounds the number of $r$-cliques in $F$-free graphs with the given number of vertices and edges. For sake of completeness, we include the argument here.

Let $d(v)$ be the degree of $v$ in $G_{\textup{blue}}$.
 Obviously the neighborhood of every vertex in $G_{\textup{blue}}$ is $F'$-free. An $F'$-free graph on $d(v)$ vertices contains at most $ex(d(v),K_{r-1},F')\le cd(v)$ copies of $K_{r-1}$.
 Thus $v$ is contained in at most $cd(v)$ copies of $K_r$ in $G_{\textup{blue}}$. If we sum, for each vertex, the number of $K_{r}$'s containing a vertex, then each $K_{r}$ is counted $r$ times. On the other hand as $\sum_{v\in V(G_{\textup{blue}})} d(v)=2|E(G_{blue})|= 2m$, we have $\sum_{v\in V(G_{\textup{blue}})} cd(v)=2cm$. This gives that the number of blue $K_r$'s is at most $2cm/r$. Thus we have
 
 \begin{eqnarray*}
 g_r(G)\le 2^{r-1}ex(n,F_0)-m +2cm/r\le \max\left\{1,\frac{2c}{r}\right\} (2^{r-1}ex(n,F_0)-m+m)=\\
 \max\left\{1,\frac{2c}{r}\right\} 2^{r-1}ex(n,F_0).\end{eqnarray*}

 The above inequality, together with Lemma \ref{celeb2} implies that $|\cH_2|\le\max\left\{1,\frac{2c}{r}\right\} 2^{r-1}ex(n,F_0)$, finishing the proof of \textbf{(ii)}.
 

 Now we show how to obtain the small improvement needed to prove Theorem \ref{cycle}. It is based on the proof of the upper bound on $ex(n,K_3,C_{2k+1})$ in \cite{as}. If $n$ is odd, replace it by $n+1$. As the stated upper bound is the same in both cases, obvious mononicity conditions show we can do this. Thus we can assume $n$ is even.
 When we take the random partition into $V_1$ and $V_2$, first we take a random partition into $n/2$ sets $U_1,\dots,U_{n/2}$ of size 2, and then randomly put one vertex into $V_1$ and the other into $V_2$. The obtained graph $G'$ will be $C_{2k}$-free, and it is divided into two components, hence it has at most $ex(|V_1|,C_{2k})+ex(|V_2|,C_{2k})$ edges. The way we chose $V_1$ ensures the above sum is $2ex(\lceil n/2\rceil,C_{2k})$. Then we can go through every step of the remaining part of the proof to obtain the result we need, if for an arbitrary $a\in A$, the probability that $a$ is in $A^*$ is still at least $1/2^{r-1}=1/4$. We will separate into cases according to the intersection of $a$ with the parts $U_i$. In case the three vertices of $a$ are in three different $U_i$'s, the probability is $1/4$. In case $a$ contains $U_i$ for some $i$, there are two cases. If $M(a)=U_i$, then the probability is $0$, otherwise it is $1/2$. As $M(a)=U_i$ happens with probability $1/3$ (having the condition that $a$ contains $U_i$), for every $i$ we have that the probability of $a$ being in $A^*$ if $a$ contains $U_i$ is $\frac{2}{3}\cdot\frac{1}{2}\ge 1/4$.
 
 This gives the first inequality of Theorem \ref{cycle}. As we have mentioned after the statement, the second inequality follows from earlier results, stated there.


\bigskip
\textbf{Acknowledgements}: Research supported by the National Research, Development and Innovation Office - NKFIH under the grants SNN 129364, KH 130371 and K 116769 and by the J\'anos Bolyai Research Fellowship of the Hungarian Academy of Sciences.

\end{document}